\documentclass{article}
\usepackage{amssymb}
\usepackage{amsmath}
\usepackage[normalem]{ulem}
\usepackage[utf8]{inputenc}

\newtheorem{theorem}{Theorem}[section]
\newtheorem{lemma}[theorem]{Lemma}
\newtheorem{proposition}[theorem]{Proposition}
\newtheorem{corollary}[theorem]{Corollary}

\newtheorem{preexample}{Example}[section]
\newenvironment{example}{\begin{preexample}}{\end{preexample}}
\newtheorem{preremark}{Remark}
\newenvironment{remark}{\begin{preremark}\rm}{\end{preremark}}

\newenvironment{proof}
{{\bf Proof:}}
{\qquad \hspace*{\fill} $\Box$}%

\newcommand{\fa}{\mathfrak{a}}%
\newcommand{\fg}{\mathfrak{g}}%
\newcommand{\fh}{\mathfrak{h}}%
\newcommand{\LC}{\mathcal{L}}%
\newcommand{\DC}{\mathcal{D}}%
\newcommand{\XC}{\mathcal{X}}%
\newcommand{\AC}{\mathcal{A}}%
\newcommand{\R}{\mathbb{R}}%

\title{Solution Curve for Linear Control Systems on Lie Groups}

\author{João Paulo Lima de Oliveira \and Alexandre J. Santana\\
Departamento de Matemática, Universidade Estadual de Maringá\\
Maringá, Brazil \and Simão N. Stelmastchuk\\
Universidade Federal do Paraná, Jandaia do Sul, Brazil}

\begin{document}
	
	\maketitle
	
	\begin{abstract}
		\noindent\hspace*{1.5cm} The purpose of this paper is to describe explicitly the solution for linear control systems on Lie groups. In case of linear control systems with inner derivations, the solution is given basically by the product of the exponential of the associated invariant  system and the exponential of the associated invariant drift field. We present the solutions in  low dimensional cases and apply  the results to obtain some controllability results.
	\end{abstract}
	
\section{Introduction}
Linear control system on $\R^n$ are control system given by differential equation
\begin{equation}\label{linearonR}
	\dot{x} = Ax + Bu, \ \ x \in \R^n,
\end{equation}
where $A$ is a $n\times n$-matrix, B is a $m \times n$, and $u = (u_1, u_2, \ldots, u_m)$ is an admissible control.  Markus in \cite{markus} studied this systems in case of  matrix groups and, later, Ayala and Tirao \cite{ayalatirao} extended for general Lie groups. Recall that  a linear control system on connected Lie group $G$ is a control system  given by the  differential equation 
\begin{equation}\label{linearonG}
	\dot{g} = \XC(g) + \sum_{i=1}^{m} u_i(t) X_i(g), \ \  g \in G,
\end{equation}
where $\XC$ is a linear vector field, namely, its flow $\varphi_t$ is a family of automorphisms of $G$, $X_1, \ldots, X_m$ are invariant vector fields, and $ u =(u_1, \ldots, u_n)$ is an admissible control. 
	
Despite linear control systems (\ref{linearonR}) have a good description of their solutions (see for instance (\cite{agrachev}), it is not true in case of systems on Lie groups (2). The first results in this way is found in \cite{ayalatirao}. Our purpose is to contribute in this direction  describing the solutions of linear control flow $\phi_t(u,.)$. In fact, we construct the solution of the linear control system (\ref{linearonG}) using a technique of Cardetti and Mittenhuber \cite{Cardetti}, that is, considering an invariant system on a semi-direct product $G\times_\varphi\R$. Then the solution curve is obtained as the integral curve of a certain invariant vector field. We also show that if the derivation $\DC$, associated to the linear vector field $\XC$,  is inner then the solution of linear control system has a simpler description (see Theorem \ref{solsemsimpgroup}). 
	
 In current literature, we can note that there are not general results on controllability problems of linear control system on Lie groups (see e.g.  \cite{ayalatirao}, \cite{Cardetti},  \cite{dasilva}, \cite{Jouan}, \cite{markus}, \cite{sanmartin}). Then with the description of the solutions we can obtain some results on  controllability.

 The paper is organized as follows, in the second section we establish some basic facts about linear controls system. In the third section we construct the solution of linear control system. In fourth section, we study controllability of linear control system on a subgroup $H$ of $G$. Finally, in the last section we study solutions of linear control system (\ref{linearonG}) when the derivation is inner, and, as application of this section, we can construct solutions on matrix groups $GL(n,\R)^+$ and on all 3-dimensional, semisimple Lie groups: $SL(2,\R)$, $SU(2)$, $SO(3,\R)$, and $SO(2,1)$.

\section{Linear Vector Fields}
This section is intended to recall some necessary facts about linear vector fields and linear systems on Lie groups that will be useful along the paper (see \cite{Jouan} for more details). Let $G$ be a connected Lie group with Lie algebra $\mathfrak{g}$. Throughout this paper, $\mathfrak{g}$ is the set of right invariant vector fields. For every $g\in G$, the maps $R_g,L_g\colon G\to G$ are, respectively, the right and left translations on $G$.

A vector field $\XC$ on $G$ is said to be \textit{linear} if its flow, which is denoted by $\varphi_t$, is an one-parameter group of automorphisms.  In \cite{Jouan} it is showed that a linear vector field $\XC$ can be characterized by one of the following conditions:
\begin{description}
	\item {(i)} for all $t \in \R$, $\varphi_t$ is an automorphism of $G$;
	\item{(ii)} for all $Y\in\fg$, $[\XC,Y]\in\fg$ and $\XC(e)=0$, where $e$ is identity of $G$.
	\item{(iii)} for all $g,h \in G$, $\XC(gh) = d(R_h)_g\XC(g) + d(L_g)_h\XC(h)$.
\end{description} 

For any linear vector field $\XC$ it is possible to associate a derivation $\DC\colon\fg\to\fg$ defined by $\DC(Y)=-[{\cal X},Y]$. From a derivation $\DC$ and a liner flow $\varphi$ given by $\XC$ we can see that 
\[
	d(\varphi_t)_e=e^{t\DC}\mbox{\hspace{1cm}and\hspace{1cm}}\varphi_t(\exp(Y))=\exp(e^{t\DC}Y).
\]
A particular case of derivations is inner derivation, that is, $\DC=-ad(X)$ for some $X\in\fg$. In this case, the linear vector field $\XC$ can be decomposed as $\XC=X+dIX$, where $dIX$ is the left-invariant vector field induced by inverse map $I(g)=g^{-1}$.

A linear system on $G$ is a control system of the form
\[
	\Sigma_L\colon\displaystyle\frac{dg}{dt}={\cal X}(g)+\displaystyle\sum_{j=1}^{m} u_{j}Y_j(g),
\]
where ${\cal X}$ is a linear vector field, $Y_1,\ldots, Y_m$ are right-invariant vector fields and the control functions $u\colon\R\to U\subset\R^m$ belong to a subset ${\cal U}\subset L_{loc}^{\infty}(\R;\R^m)$ of the space of the locally integrable functions. Without loss of generality we can assume, throughout this paper, that  ${\cal U}$ is the set of piece-wise constant functions.

For $u\in{\cal U}$ and $g\in G$, we denote the solution of $\Sigma_L$ starting at $g$ associated to $u$ by $\phi_t(u,g)$ with $t\in\R$. For $g,h\in G$, we say that $h$ is reachable from $g$ in time $t$ if there is $u$ such that $\phi_t(u,g)=h$. Let us denote by ${\cal A}_t(g)$ the set of all points in $G$ reachable from $g$ in time $t$. The reachable set from $g$ is defined as follows
\[
	{\cal A}(g)=\bigcup_{t\geq0}{\cal A}_t(g).
\]
	
The system $\Sigma_L$ is said to be controllable if any $h$ is reachable for any $g$. Equivalently, ${\cal A}(g)=G$ for any $g \in G$.

\section{Solution for Linear Control Systems}
In this section we come up with a construction presented in \cite{Cardetti} which associates to a linear control system $\Sigma_L$ an invariant one, denoted by $\Sigma_I$. We begin by remembering that the flow $\varphi_t$ of the linear vector field ${\cal X}$ yields a representation $\varphi\colon\R\rightarrow Aut(G)$.  This allows us to define the semi-direct product $G\times_{\varphi}\R$, that is, the set $G\times\R$ endowed with the product $(g,t)(h,s)=(g\varphi_{t}(h),t+s)$ (see e.g. \cite{San Martin}). It is well-know that $G\times_{\varphi}\R$ is a Lie group. Furthermore, its correspondent Lie algebra is the semi-direct product of Lie algebras $\mathfrak{g}\times_{\sigma}\R$, where $\sigma\colon\R\rightarrow Der(\mathfrak{g})$ is defined as
\[
	\sigma(t)(Y)=ad_{t{\cal X}}(Y)=[t\XC,Y].
\]
The relation between $\varphi$ and $\sigma$ is given by $d\varphi_0=\sigma$.

For any vector fields $(Y,t),(W,s) \in \fg \times_{\sigma}\R$, the Lie bracket is given by the formula
\[
	\left[(Y,t)(W,s)\right]=\left([Y+t{\cal X},W+s{\cal X}],0\right).
\]
Throughout this paper, for $(g,r)\in G\times_{\varphi}\R$, $R_{(g,r)}$ we denote the right translation and, when not specified, the differential $dR_{(g,r)}$ is evaluated at the group identity.

Now we determine the value of a vector field $(W,s)$ on an arbitrary point $(g,r)$.
\begin{proposition}\label{Efcamp}
	If $(W,s)\in\mathfrak{g}\times_\sigma\R$ and $(g,r)\in G\times_\varphi\R$, then 
	\[
		(W,s)(g,r)=\left(W(g)+s{\cal X}(g),s\right).
	\]
\end{proposition}
\begin{proof}
	We first write $(W,s)(g,r)=dR_{(g,r)}(W,s)$. Thus, in matrix notation, the differential $dR_{(g,r)}$ gives
	\begin{eqnarray*}
		dR_{(g,r)}(W,s)=\left(
		\begin{array}{cc}
			dR_g&\XC(g)\\
			0&1				
		\end{array}\right)
		\left(
		\begin{array}{cc}
			W\\
			s
		\end{array}\right)=\left(W(g)+s{\cal X}(g),s\right).
	\end{eqnarray*}
\end{proof}

This Proposition allow us to describe exponentials of invariant vector fields of $\mathfrak{g}\times_{\sigma}\R$.

\begin{lemma}
	If $(W,0)$, $(0,s)\in\mathfrak{g}\times_{\sigma}\R$, then their exponentials are smooth curves $(\exp(tW),0)$ and $(e,st)$, respectively, with $t\in\R$.
\end{lemma}
\begin{proof}
	We first compute the exponential for $(W,0)$. To this purpose, we note that $(W,0)(g,r)=(W(g),0)$ for all $(g,r)\in G\times_{\varphi}\R$ in view of Proposition \ref{Efcamp}. By definition of exponential, 
	\[
		\dfrac{d}{dt}(\exp(tW),0)=\left(\dfrac{d}{dt}\exp(tW),0\right)=\left(W(\exp(tW)),0\right)=(W,0)(\exp(tW),0).
	\]
	The result follows by uniqueness of solution. Analogously we can see that the curve $(e,st)$ is the exponential of $(0,s)$. 
\end{proof}

In the following, the previous Lemma will be used to determine the exponential of an invariant vector field $(W,s) \in \fg \times_{\sigma}\R$.

\begin{proposition}\label{FormExp}
	Let $(W,s)$ be an invariant vector field on $G \times_{\sigma} \R$. It follows that
	\begin{equation}\label{formexp}
		\exp(t(W,s))=\left(\displaystyle\lim_{n\rightarrow\infty}\displaystyle\prod_{i=0}^{n-1}\varphi_{ist/n}\circ\exp(t/n\cdot W),st\right).
	\end{equation}
\end{proposition}
\begin{proof}
	We first write $(W,s)=(W,0)+(0,s)$. Applying the Lie product formula gives
	\begin{eqnarray*}
	\exp(t(W,s))&=&\displaystyle\lim_{n\rightarrow\infty}\left(\exp(t/n(W,0))\cdot\exp(t/n(0,s))\right)^n.
	\end{eqnarray*}
	Using the above Lemma and the semidirect product we see that
	\begin{eqnarray*}
		\exp(t(W,s))
		&=&\displaystyle\lim_{n\rightarrow\infty}\left((\exp(t/n\cdot W),0)(e,st/n)\right)^n\\
		&=&\displaystyle\lim_{n\rightarrow\infty}(\exp(t/n\cdot W),st/n)^n\\
		&=&\left(\displaystyle\lim_{n\rightarrow\infty}\displaystyle\prod_{i=0}^{n-1}\varphi_{ist/n}\circ\exp(t/n\cdot W),st\right),
	\end{eqnarray*}
	and the proof is complete.
\end{proof}

Using the relation $d(\varphi_t)_e=e^{t\DC}$ we can rewrite formula (\ref{formexp}) as 
\[
	\exp(t(W,s))=\left(\displaystyle\lim_{n\rightarrow\infty}\displaystyle\prod_{i=0}^{n-1}\exp\left(t/n\cdot e^{\DC_t} W\right),st\right),
\]
where, for simplification, we denote $\DC_t=\dfrac{ist}{n}\DC$.

Consider the vector fields $\bar{{\cal X}}=(0,1)$, $\bar{Y_j}=(Y_j,0) \in \fg \times_{\sigma}\R$, for each $j=1,\ldots,m$. It follows from Proposition \ref{Efcamp} that these vector fields can be expressed as $\bar{{\cal X}}(g,r)=({\cal X}(g),0)$ and $\bar{Y}_j(g,r)=(Y_j(g),0)$. We define the following invariant control system on $G\times_{\varphi}\R$:
\[
	\Sigma_I\colon\displaystyle\frac{d(g,r)}{dt}=\bar{{\cal X}}(g,r)+\displaystyle\sum_{j=1}^{m} u_{j}\bar{Y_j}(g,r).
\]
Equivalently, we have
\[
	\left(\begin{array}{c}
		dg/dt\\
		dr/dt
	\end{array}\right)=
	\left(\begin{array}{c}
		{\cal X}(g)+\displaystyle\sum_{j=1}^{m} u_{j}Y_j(g)\\
		1
	\end{array}\right).
\]
This system was built in such a way that $\pi(\Sigma_I)=\Sigma_L$, where $\pi\colon G\times_{\varphi}\R\to G$ is the projection on the first coordinate. If we denote ${\cal A}_I(g,r))$  the reachable set of $\Sigma_I$ at point $(g,r)$ and ${\cal A}_L(g)$ the reachable set of $\Sigma_L$ at point $g$, then $\pi({\cal A}_I(g,r))={\cal A}_L(g)$. Furthermore, the invariance of the system allows us to write ${\cal A}_I(g,r)={\cal A}_I(e,0)\cdot(g,r)$. And now, we move on to the main result.

\begin{theorem}\label{Formgersol}
	 For $u=(u_1,\ldots,u_m)\in\R^m$ the curve 
	\begin{equation}\label{sollysys}
		\phi_t(u,e)=\displaystyle\lim_{n\rightarrow\infty}\displaystyle\prod_{i=0}^{n-1}\varphi_{it/n}\circ \exp\left(\dfrac{t}{n}\displaystyle\sum_{j=1}^{m}u_jY_j\right)
	\end{equation}
	is the solution, with initial condition $\phi_0(u,e) = e$, of the dynamical system
	\[
		\Sigma_L\colon\displaystyle\frac{dg}{dt}={\cal X}(g)+\displaystyle\sum_{j=1}^{m} u_{j}Y_j(g).
	\]
\end{theorem}
\begin{proof}
	Let us denote $W=\displaystyle\sum_{j=1}^{m}u_jY_j$ and $\exp(t(W,1))=(\phi_t(u,e),t)$. From Proposition \ref{Efcamp} we have that
	\[
		(W,1)(\phi_t(u,e),t)=(W(\phi_t(u,e))+{\cal X}(\phi_t(u,e)),1)=\left({\cal X}(\phi_t(u,e))+\displaystyle\sum_{j=1}^{m} u_{j}Y_j(\phi_t(u,e)),1\right).
	\]
	On the other side, the curve $(\phi_t(u,e),t)$ is the integral curve of $W$. Therefore  
	\[
		(W,1)\exp(t(W,1))=(W,1)(\phi_t(u,e),t)=(\frac{d\phi_t(u,e)}{dt},1).
	\]
	 We thus get 
	\[
		\left(\begin{array}{c}
			\frac{d\phi_t}{dt}(u,e)\\
			1
		\end{array}\right)=
		\left(\begin{array}{c}
			{\cal X}(\phi_t(u,e))+\displaystyle\sum_{j=1}^{m} u_{j}Y_j(\phi_t(u,e))\\
			1
		\end{array}\right).
	\]
	Taking the projection on the first coordinate we conclude that the curve $\phi_t(u,e)$ satisfies the differential equation of the dynamical system. As $\phi_0(u,e)=e$ we have that this is the solution of the system at the identity. On the other side, Proposition \ref{FormExp} give us a description of $\exp(t(W,1))$. Since $\exp(t(W,1))=(\phi_t(u,e),t)$, we conclude that  
	\[
		\phi_t(u,e) =\displaystyle\lim_{n\rightarrow\infty}\displaystyle\prod_{i=0}^{n-1}\varphi_{it/n}\circ \exp\left(\dfrac{t}{n}\displaystyle\sum_{j=1}^{m}u_jY_j\right).
	\]
\end{proof}

The above theorem allow us to recover a well-know result about linear systems. 

\begin{corollary}\label{Solarbpoint}
	Let $g\in G$ be an arbitrary point and consider the linear dynamical system as in the previous theorem. The solution $\phi_t(u,g)$ of the system starting at $g$ is given by the formula 
	\[
		\phi_t(u,g) = \phi_t(u,e) \varphi_t(g).
	\]
\end{corollary}
\begin{proof}
	Consider an arbitrary point $(g,r)\in G\times_{\varphi}\R$. Denote $\phi_I((g,r),t)$ the solution of control system $\Sigma_I$ at point $(g,r)$. We have already seen that $\phi_L(g,t)=\pi(\phi_I((g,r),t))$. Now, using the right invariance property and the semidirect product we get
	\[
		\phi_t(u,g)=\pi\left(\phi_I((e,0),t)(g,r)\right).
	\]
	By the proof of the previous theorem, it follows that $\phi_I((e,0),t)=(\phi_t(u,e),t)$. Therefore 
	\[
		\phi_t(u,g)=\pi(\phi_t(u,e)\varphi_t(g),t+r) = \phi_t(u,e)\varphi_t(g) .
	\]

\end{proof}
	
\begin{remark}\label{solarbtime}
	It is possible to apply the previous theorem (and its corollary) to obtain the solution curve for linear control systems by means of the cocycle property. In fact, without loss of generality, consider a piece-wise constant control $u\colon[0,t+s]\to U\subset\R^m$, with $t,s\in\R$, given by concatenation 
	\[
		u(r) = \left\{
		\begin{array}{ll}
			u_1 &, \mbox{if } \,\, r \in [0,t]\\
			u_2 &, \mbox{if } \,\, r \in (t, t+s],
		\end{array}
		\right.
	\]
	where $u_1$ and $u_2$ are constants. Now, using Corollary \ref{Solarbpoint} and cocycle property we get
	\[
		\phi_{t+s}(u,e)=\phi_s(u_2,e)\varphi_s(\phi_t(u_1,e)).
	\]
\end{remark}

%	Note that as the control functions $u$ and $\theta_tu$ are constants on the respective intervals, formula (\ref{sollysys}) can be applied to express the solution above.

\begin{example}[{\bf Linear Control Systems on Heisenberg Group}]
	Let $G$ be the \textit{Heisenberg} group, that is, the set of all real matrix of the form
	\[
		\begin{pmatrix}
			1&x&z\\
			0&1&y\\
			0&0&1
		\end{pmatrix}.
	\]
	As usual, we identify this group with $\R^3$ in such a way that the group product is defined as 
	\[
		(x_1,y_1,z_1)\cdot (x_2,y_2,z_2)=(x_1+x_2,y_1+y_2,z_1+z_2+x_1y_2).
	\]
	In this case, the Lie algebra of the Heisenberg group is the vector space $\R^3$ with the Lie bracket defined as $[(x_1,y_1,z_1)(x_2,y_2,z_2)]=(0,0,x_1y_2-x_2y_1)$ and the exponential map $\exp\colon\R^3\to\R^3$ as
	\[
		\exp(x,y,z)=\left(x,y,\dfrac{xy}{2}+z\right).
	\]
	The right invariant vector fields $Y=(m,n,p)$ in $G$ have the form
	\[
		Y(x,y,z)=m\dfrac{\partial}{\partial x}+n\dfrac{\partial}{\partial y}+(my+p)\dfrac{\partial}{\partial z},
	\]
	while the matrix of a derivation $\DC$ associated to a linear vector field ${\cal X}$ is written as
	\[
		\DC=
		\begin{pmatrix}
			a_{11}&a_{12}&0\\
			a_{21}&a_{22}&0\\
			a_{31}&a_{32}&a_{11}+a_{22}
		\end{pmatrix}.
	\]
	We consider the following linear control system on $G$
	\[
		\displaystyle\frac{dg}{dt}={\cal X}(g)+uY(g),
	\]
	where $Y=(0,0,p)$ and ${\cal X}$ is linear vector field associated to derivation $\DC$. 
	
	In view of Remark \ref{solarbtime}, it is enough to express the solution of the system on an interval $[0,T]$ in which the control function $u$ is constant. In formula (\ref{sollysys}),  we note, by a direct calculation, that
	\[
		\dfrac{t}{n}ue^{it/n\cdot\DC}Y=\left(0,0,up\dfrac{t}{n}\right).
	\]
	Taking the exponential we obtain
	\begin{equation}\label{solutionheisenberg}
		\phi_t(u,e) = \left(0,0,\displaystyle\lim_{n\rightarrow\infty}\displaystyle\sum_{i=0}^{n-1}up\dfrac{t}{n}\right)=\left(0,0,\displaystyle\int_{0}^{t}upds\right)=(0,0,upt),
	\end{equation}
	where $u(t)=u$, for $0\leq t\leq T$. 
	
	From the solution  (\ref{solutionheisenberg}) it is easy to see  that the linear control system is not controllable on $G$. 
\end{example}
	
	\section{Controlability}
	
	In this section, our intention is to study controllability of a linear system in a Lie subgroup of $G$. We begin by establishing some current notation:
	\begin{eqnarray*}
		\LC & = & \operatorname{Lie}\{ \XC, Y_1, \ldots, Y_m\}\\
		\fa & = & \operatorname{Lie} \{ Y_1,\ldots, Y_m\}\\
		\DC\fa & = & \operatorname{span}\{ \DC^i(Y):\, Y \in \fh, i \in \mathbb{N} \}.
	\end{eqnarray*}

	In general the subalgebra $\fa$ is not ${\cal D}$-invariant. Consider, for example, a linear system 
	\[
		\dfrac{dx}{dt}=Ax+ub
	\]\newline
	defined on $\R^2$ where $A$ is a rotation matrix. It is clear that $[A,b]=-Ab \notin \fa$. For this reason, we consider $\fh = {\cal LA}(\DC\fa)$ the $\fg$ subalgebra generated by $\DC\fa$ and recall Proposition 3 from \cite{Jouan}:
	\begin{proposition}\label{propjouan}
		The subalgebra $\fh$ is $\DC$-invariant. Therefore, the Lie algebra $\LC$ satisfies
		\[
		\LC=\R{\cal X}\oplus \fh.
		\]
		Moreover, the linear system satisfies the rank condition if and only if $\fh=\fg$.
	\end{proposition}
	
	We denote by $H$ the closed subgroup generated by $\fh = {\cal LA}(\DC\fh)$. It follows that $\varphi_t(H) \subset H$. Furthermore, $\XC$ is tangent to $H$, in other words, $\XC(h) \in T_h H$ if $h \in H$.
	
	We now establish a first result about control system on $H$.
	\begin{proposition}
		If $h \in H$, then for all control $u$ and for all $t$ we have $\phi_t(u, h) \in H$. 
	\end{proposition}
	\begin{proof}
		Corollary \ref{Solarbpoint} gives that $\phi_t(u,h) = \phi_t(u,e) \varphi_t(h)$. We know that $\varphi_t(h) \in H$ for $h \in H$, since $H$ is $\varphi_t$-invariant. From solution (\ref{sollysys}) we can see that $\varphi_t(u,e) \in H$. Therefore, $\phi_t(u,h) \in H$.
	\end{proof}
	
	Here, the key of proof is the fact $\phi_t(u,e) \in H$. This fact, given by Theorem \ref{Formgersol}, allow us to study  controllability  in a new perspective. For example, a direct result about controllability is obtained. 
	\begin{theorem}
		If $H \not\subseteq G$, then $\Sigma_L$ is not controllable on $G$. Equivalently, if $\fh\not\subseteq \fg$, then $\Sigma_L$ is not controllable on $G$. 
	\end{theorem}
	\begin{proof}
		We begin by observing that $\varphi_t(u,h) \in H$ when $h \in H$ by Proposition above. Thus, for any $g \in G\backslash H$, it is not possible to find some control and time $t>0$ such that $\varphi_t(u,h) =g$. It means that $\Sigma$ is not controllable on $G$.
	\end{proof}

	Theorem above shows a restriction to controllability of $\Sigma$ in $G$. However, it does not show any information about controllability of $\Sigma$ on $H$. Our idea is to give controllability conditions on $H$. Our first step is to point out that the rank condition occur naturally. 

	\begin{corollary}
		The linear system on $H$ satisfies the rank condition.
	\end{corollary}
	\begin{proof}
		It is a consequence of the definition of $H$ and Proposition \ref{propjouan} since $\fh$ is the Lie algebra of $H$.
	\end{proof}

	To show our first results we need to introduce the stable, unstable and central subgroup to linear system on $H$. Since $\fh$ is $\DC$-invariant, it follows that $\DC$ is a derivation on $\fh$. From eigenvalues of derivation $\DC$ on $\fh$ we can written 
	\[
	\fh^+=\bigoplus_{\alpha; \mathrm{Re}(\alpha)>0}\fh_{\alpha}, \;\;\;\;\fh^0=\bigoplus_{\alpha; \mathrm{Re}(\alpha)=0}\fh_{\alpha}, \;\;\;\mbox{ and }\;\;\;\fh^-=\bigoplus_{\alpha; \mathrm{Re}(\alpha)<0}\fh_{\alpha},
\]
where $\alpha$ are eigenvalues of the derivation $\DC$, such that 
\[
	\fh =\fh^+\oplus \fh^{0} \oplus \fh^{-} \ \ \mbox{and} \ \ [\fh_{\alpha},\fh_{\beta}] = \fh_{\alpha + \beta},
\]
with $\alpha + \beta = 0$ if the sum is not an eigenvalue. Let us denote by $H^+$, $H^0$ and $H^-$ the connected $\varphi_t$-invariant Lie subgroups $H^+$, $H^0$ and $H^-$ with Lie algebras $\fh^+$, $\fh^0$ and $\fh^-$, respectively.

To make the next proof clear, we remember the notation of attainable set from identity $e$: 
\[
	\AC_H := \{ x \in H: \phi_t(u,e) = x, \, \mbox{for some}\,  t \geq 0 \}.
\]
If we change $\XC$ by $-\XC$ we obtain the inverse linear flow $\varphi_t^*$, which satisfies the following property $\varphi_t^* = \varphi_{-t}$. Thus, we have the inverse linear control system defined by 
\[
	\displaystyle\frac{dg}{dt}=-{\cal X}(g)+\displaystyle\sum_{j=1}^{m} u_{j}Y_j(g).
\]
Its attainable set from identity is denoted by 
\[
	\AC_H^* := \{ x \in H: \phi_t^*(u,e) = x, \, \mbox{for some}\,  t \geq 0 \}.
\]

\begin{theorem}
	Let $G$ be a solvable Lie group and $H$ the Lie subgroup of $G$ given above. If a derivation $\DC$ of linear control field $\XC$ has only eigenvalues with zero real part when it restricts to $H$, then $\Sigma_L$ is controllable on $H$ . 
\end{theorem}
\begin{proof}
	We first observe that $H$ is solvable. Furthermore, it is true that $\Sigma_L$ satisfies ad-rank on $H$. Thus, Proposition 2.13 in \cite{dasilva} assures that $\AC_H$ and $\AC^*_H$ are open. Also, from Theorem 4.1 of \cite{dasilva} we conclude that $\Sigma_L$ is controllable in $H$.
\end{proof}

\begin{theorem}
	Let $G$ be a nilpotent Lie group and $H$ the Lie subgroup of $G$ given above. Assume that $\Sigma$ is bounded. Then $\Sigma_L$ is controllable on $H$ if and only if $H =H_0$.
\end{theorem}
\begin{proof}
	It is a direct application of Theorem 4.5 in \cite{dasilva}.
\end{proof}

Other way to study controllability is by means of rank condition and Lie saturate, as presented at Theorem 12 in \cite[ch.3]{jurdjevic}. We show that $\DC$-invariance is sufficient condition to controllability of $\Sigma_L$ in $H$. In fact, from Proposition 4 in \cite{Jouan} we have that $\fa$ is included on Lie Saturate. Then a possible condition to controllability is $\fa = \fh$. Our next result show a condition to occur this equality.

\begin{proposition}
	$\fa = \fh$ if and only if $\fa$ is $\DC$-invariant.
\end{proposition}
\begin{proof}
	We first suppose that $\fa$ is $\DC$-invariant. Thus, it is direct that $\DC\fa \subset \fa$. Consequently, $\fh \subset \fa$. We thus get $\fh = \fa$. 
	
	Conversely, suppose that $\fa = \fh$. Then it is true that $\DC(Y) \in \fa$ for all $Y \in \fa$. It means that $\fa$ is $\DC$-invariant.  
\end{proof}

\begin{theorem}
	If $\fa$ is $\DC$-invariant, then $\Sigma$ is controllable on $H$.
\end{theorem}
\begin{proof}
	If $\fa$ is $\DC$-invariant, then $\fa =\fh$. From Proposition 4 in \cite{Jouan} it follows that $\fh$ is the Lie saturate. On the other hand, the control flow $\phi_t$ satisfies the rank-condition. Therefore, by Theorem 12 in \cite{jurdjevic}, the linear control system is controllable in $H$.
\end{proof}

%==============================================

\section{Inner Derivation Case}

In this section we apply formula (\ref{sollysys}) when $\DC$ is inner, namely, there is $X\in\fg$ such that $\DC=ad(X)$ (see for instance \cite{San Martin}). In particular, the linear vector field $\XC$ can be decomposed as $\XC=X+dIX$, where $X$ is a right invariant vector field and $dIX$ is the vector field induced by $I(g)=g^{-1}$. These facts allow us to describe the solution in a simpler way and to relate it to the solution of an associated invariant system. We begin by remembering, in this case, that the flow $\varphi_t$ can be written as $\varphi_t(g) = \exp(tX)g\exp(-tX)$ (see e.g. \cite{Ayala} or \cite{Jouan}). Note that the results stated along this section are, in particular, true for semisimple Lie groups since all derivations defined on their Lie algebras are inner (see for instance \cite{San Martin}).	

Given a linear control system $\Sigma_L$, we yield the following right-invariant control system
\[
	\Sigma_I\colon\dfrac{dg}{dt}=X(g)+\displaystyle\sum_{j=1}^mu_jY_j(g),
\]
where $X$ is such that $\XC=X+dIX$.

Replacing $\varphi_t(g)=\exp(tX)g\exp(-tX)$ in the formula (\ref{sollysys}) of Theorem \ref{Formgersol}, we obtain the following description for the solutions of $\Sigma_L$:

\begin{theorem}\label{solsemsimpgroup} 
	Let $u$ be an admissible control function and suppose that it is constant on an interval $[0,T]$. Then the solution of the linear system with control function $u$ is the curve
	\begin{equation}\label{solsemsimp}
		\phi_t(u,e)= \exp\left(tX+\displaystyle\sum_{j=1}^mtu_jY_j\right)\exp(-tX), \mbox{\hspace{0.5cm}} 0\leq t\leq T.
	\end{equation}
\end{theorem}
\begin{proof}
	We first simplify the notation by denoting $Y=\displaystyle\sum_{j=1}^mu_jY_j$. Replacing 
	\[
		\varphi_{it/n}\left(\exp\dfrac{t}{n}Y\right)=\exp\left(\dfrac{it}{n}X\right) \exp\left(\dfrac{t}{n}Y\right)\exp\left(-\dfrac{it}{n}X\right)
	\]
	in formula (\ref{sollysys}) we see that
	\begin{eqnarray*}
		\phi_t(u,e)
		& = &\displaystyle\lim_{n\rightarrow\infty}\displaystyle\prod_{i=0}^{n-1}\varphi_{it/n}\left(\exp\dfrac{t}{n}Y\right)\\
		& = &\displaystyle\lim_{n\rightarrow\infty}\displaystyle\prod_{i=0}^{n-1}\exp\left(\dfrac{it}{n}X\right) \exp\left(\dfrac{t}{n}Y\right)\exp\left(-\dfrac{it}{n}X\right)\\
		& = &\displaystyle\lim_{n\rightarrow\infty}\left(\exp\dfrac{t}{n}Y\exp\dfrac{t}{n}X\right)^{n-1}\exp\left(\dfrac{t}{n}Y\right)\exp\left(\dfrac{(1-n)t}{n}X\right).
	\end{eqnarray*}
	Inserting $\exp\left(\dfrac{t}{n}Y\right)\exp\left(\dfrac{t}{n}X\right)\exp\left(-\dfrac{t}{n}X\right)\exp\left(-\dfrac{t}{n}Y\right)$ in the above expression  we can assert that
	\[
		\phi_(u,e) = \displaystyle\lim_{n\rightarrow\infty}\left(\exp\dfrac{t}{n}Y\exp\dfrac{t}{n}X\right)^{n}\exp\left(-tX\right).
	\]
	Using the  Lie product Formula we conclude that
	\[
		\phi_t(u,e) = \exp\left(tX+\displaystyle\sum_{j=1}^mtu_jY_j\right)\exp(-tX),
	\]
	and proof is complete.
\end{proof}

%\begin{theorem}
%	If $\DC$ is an inner derivation on Lie group $G$, then $\Sigma$ is controllable on $H$. 
%\end{theorem}
%\begin{proof}
%	We first observe that $\Sigma_L$ satisfies the rank condition on $H$. By Theorem 7.1 in \cite{JurdSuss}, the right invariant system $\Sigma_I$ is controllable on $H$. Furthermore, since ad-rank is satisfied, then it is true that $\Sigma$ is controllable on $H$ because Theorem 1 in \cite{Jouan}.
%\end{proof}

\subsection{Application}

This section presents applications of the above results in case of  linear systems with inner derivation. Initially, we consider a system defined in $GL(n;\R)^+$ in the same way as in \cite{markus}, and describe the solution curve for this case. According to \cite{Biggs}, there are four three-dimensional semisimple Lie groups, which are $Sl(2)$, $SU(2)$, $SO(3)$ and $SO(2,1)_0$, so we construct the solutions of linear systems in these Lie groups.

\subsection{Lie group in $GL(n;\R)^+$}
	Let $G=Gl(n;\R)^+$ be the set of all $n\times n$ real matrices with positive determinant and $\fg=\mathfrak{gl}(n;\R)$ its Lie algebra. For $A \in \fg$, the vector field ${\cal X}_A(g)=Ag-gA$ is linear and its associated flow is $\varphi_t(g)=e^{tA}ge^{-tA}$. Given $B_1,\ldots,B_m \in \fg$, consider the right-invariant fields $B_j(g)=B_jg$ and the linear control system
	\[
		\dfrac{dg}{dt}=Ag-gA+\displaystyle\sum_{j=1}^mu_jB_j(g).
	\]
	
	Applying formula (\ref{sollysys}) for a constant control $u$, we can write the solution at the identity we can write the solution as
	\[
		\phi_t(u,e) = e^{t\left(A+\sum u_jBj\right)}e^{-tA}.
	\]
	Now,  considering a concatenation of two constant controls $u_1$ and $u_2$ on an interval $[0,t+s]$, as in Remark \ref{solarbtime}, the solution becomes
	\[
		\phi_t(u,g) = e^{s\left(A+\sum u_{2j}Bj\right)}e^{t\left(A+\sum u_{1j}Bj\right)}e^{-tA},
	\]
	where $u_1 =(u_{11}, \ldots, u_{1n})$ and $u_2 =(u_{21}, \ldots, u_{2n})$.

	\begin{corollary}
		If $\operatorname{span}\{B_1,B_2, \ldots, B_n\}$ is invariant by $A$, then the control system is controllable on $\operatorname{span}\{B_1,B_2, \ldots, B_n\}$.
	\end{corollary}

\subsection{Semisimple Lie groups of dimension 3}
We only study linear system on special linear group $Sl(2,\R)$ because in others are the same accounts. Its Lie algebra have the following description
\[
	\mathfrak{sl}(2,\R)
	=\left\{
	\begin{pmatrix}
		x&y\\
		z&-x
	\end{pmatrix}:x,y,z\in\R
	\right\}.
\]
From above description, we have that the eigenvalues $\lambda_1$, $\lambda_2$ of an element $Z\in\fg$ satisfy the condition $\lambda_1+\lambda_2=0$.
	
Given a linear system $\Sigma_L$ over $Sl(2,\R)$, for a control function $u$, constant on an interval $[0,T]$, we denote
\[
	\Sigma_I=X+\displaystyle\sum_{j=1}^mu_j(t)Y_j, \mbox{\hspace{0.5cm}} 0\leq t\leq T, 
\]
the matrix of the right-invariant associated system. With these notations, by Theorem \ref{solsemsimpgroup}, the solution of $\Sigma_L$ is
\[
	\phi_t(u,e)=\exp\left(t\Sigma_I\right)\exp(-tX), \mbox{\hspace{0.5cm}} 0\leq t\leq T.
\]
We denote by $\lambda_1$, $\lambda_2$ the eigenvalues of $\Sigma_I$ and by $\delta_1$, $\delta_2$ the eigenvalues of $X$. Applying methods of functional calculus we have $\exp(t\Sigma_I)=p(\Sigma_I)$ and $\exp(tX)=q(X)$, where $p(z)=az+b$ and $q(z)=cz+d$ are polynomials satisfying the following conditions:
	\[
		\begin{array}{l}
			p(\lambda_i)=a\lambda_i+b=e^{t\lambda_i},\mbox{ for } i=1,2.\\
			q(\delta_j)=c\delta_j+d=e^{t\delta_j},\mbox{ for }j=1,2.
		\end{array}
	\]
	Using the fact that $\lambda_1+\lambda_2=\delta_1+\delta_2=0$ we obtain 
	\[
		p(z)=\dfrac{senh(t\lambda)}{\lambda}z+cosh(t\lambda)\,\,\,\, \mbox{and}\,\,\,\, q(z)=\dfrac{senh(t\delta)}{\delta}z+cosh(t\delta),
	\]
	where $\lambda$ and $\delta$ may be any of the correspondent eigenvalues. The solution of the linear system then is written as
	\[
		\phi_(u,e) = \left(\dfrac{senh(t\lambda)}{\lambda}\Sigma_I+cosh(t\lambda)Id\right)\left(\dfrac{senh(t\delta)}{\delta}X+cosh(t\delta)Id\right),
	\]
	for $0\leq t\leq T$.
	
	As mentioned above, on the other 3-dimensional, semisimple Lie groups $SU(2)$, $SO(3)$ and $SO(2,1)$, the argument to construct solutions of linear control systems are the same. Summarizing:  
	\begin{itemize}
		\item In $Sl(2,\R)$, \\
		$\phi_t(u,e) =\left(\dfrac{senh(t\lambda)}{\lambda}\Sigma_I+cosh(t\lambda)Id\right)\left(\dfrac{senh(t\delta)}{\delta}X+cosh(t\delta)Id\right)$;
		\item In $SU(2)$, \\
		$\phi_t(u,e) =\left(\dfrac{sen(t\lambda)}{\lambda}\Sigma_I+cos(t\lambda)Id\right)\left(\dfrac{sen(t\delta)}{\delta}X+cos(t\delta)Id\right)$;
		\item In $SO(3)$ or $SO(2,1)_0$,\\ 
		\begin{eqnarray*}
			\phi_t(u,e) = \left(\dfrac{cosh(t\lambda)-1}{\lambda^2}\Sigma_I^2 \right. &+ & \left.\dfrac{senh(t\lambda)}{\lambda}\Sigma_I+Id\right) \cdot \\
			& & \left(\dfrac{cosh(t\delta)-1}{\delta^2}X^2+\dfrac{senh(t\delta)}{\delta}X+Id\right).
		\end{eqnarray*}
	\end{itemize}

%	\begin{table}[htbp]
%		\begin{tabular}{|c|c|} \hline
%			Lie subgroup & $\phi_t(u,e)$  \\ \hline
%			$Sl(2,\R)$ & $\left(\dfrac{senh(t\lambda)}{\lambda}\Sigma_I+cosh(t\lambda)Id\right)\left(\dfrac{senh(t\delta)}{\delta}X+cosh(t\delta)Id\right)$ \\ \hline
%			$SU(2)$    & $\left(\dfrac{sen(t\lambda)}{\lambda}\Sigma_I+cos(t\lambda)Id\right)\left(\dfrac{sen(t\delta)}{\delta}X+cos(t\delta)Id\right)$ \\ \hline
%			$SO(3)$, $SO(2,1)_0$ & $\left(\dfrac{cosh(t\lambda)-1}{\lambda^2}\Sigma_I^2+\dfrac{senh(t\lambda)}{\lambda}\Sigma_I+Id\right)\left(\dfrac{cosh(t\delta)-1}{\delta^2}X^2+\dfrac{senh(t\delta)}{\delta}X+Id\right)$ \\ \hline
%		\end{tabular}
%	\end{table}

\end{document}